\documentclass{amsart}

\newtheorem{theorem}{Theorem}[section]
\newtheorem{lemma}[theorem]{Lemma}
\newtheorem{corollary}[theorem]{Corollary}
\theoremstyle{definition}

\theoremstyle{remark}
\newtheorem{remark}[theorem]{Remark}

\numberwithin{equation}{section}

\begin{document}

\title{Zeta functions of spheres and real projective spaces}

\author{Lee-Peng Teo}
\address{Department of Applied Mathematics, University of Nottingham Malaysia Campus, Jalan Broga, 43500, Semenyih, Selangor, Malaysia.}

\email{LeePeng.Teo@nottingham.edu.my}



\date{\today}



\begin{abstract}
The Minakshisundaram-Pleijel zeta function $Z_k(s)$ of a $k$-dimensional sphere $S^k$, $k\geq 2$, is given by the Dirichlet series
\begin{equation*}
Z_k(s)=\sum_{n=1}^{\infty} \frac{P_k(n)}{\left[n(n+k-1)\right]^s}
\end{equation*}when $\displaystyle\text{Re}\,s>k/2$. Here $(k-1)!P_k(n)=(2n+k-1)(n+1)\ldots (n+k-2)$. In \cite{2}, it was shown that $Z_k(s)$ can be expressed in terms of the Hurwitz zeta function $\displaystyle \zeta(s;a)=\sum_{n=0}^{\infty}(n+a)^{-s}$ in the following way:
\begin{equation*}
Z_k(s)=\frac{1}{(k-1)!}\sum_{l=0}^{\infty} (-1)^l \begin{pmatrix} -s\\l\end{pmatrix}\left(\frac{k-1}{2}\right)^{2l}
\sum_{j=0}^{k-1}B_{k,j}\zeta\left(2s+2l-j;\frac{k+1}{2}\right),
\end{equation*} where
\begin{equation*}
B_{k,j}=\sum_{p=0}^{k-j-1}(-1)^{k+j+1}\begin{pmatrix} j+p\\j\end{pmatrix}\left(\frac{k-1}{2}\right)^p\left(s_{k,j+p+1}+s_{k-1,j+p}\right)
\end{equation*}is expressed in terms of the Stirling numbers of the first kind $s_{n,k}$.

In this work, we gave another much simpler derivation of the formula above, and prove that $B_{k,j}$ is the coefficient of $x^j$ of a certain polynomial  of degree $(k-1)$. The polynomial is odd if $k$ is even and is even if $k$ is odd. This immediately implies that $B_{k,k-2h}=0$ if $h$ is an integer, a result proved in \cite{2} using a much more complicated method. As a byproduct, we  give a simple recursive algorithm to compute the coefficients $B_{k,j}$.

It has been shown in \cite{2} that $Z_k(s)$ at most has  simple poles at $\displaystyle s=\frac{k}{2}-n$, where $n$ is a nonnegative integer. We strengthen this result and show that when $k$ is even, $Z_k(s)$ at most has simple poles at $\displaystyle s=1,2,\ldots,\frac{k}{2}$.

As have been observed in \cite{2}, when $k$ is odd and $n$ is a positive integer, $Z_k(-n)=0$. We show that when $k$ is odd, $Z_k(0)=-1$. When $k$ is even and $n$ is a nonnegative integer, we derive the values of $Z_k(-n)$ as a finite sum over $B_{k,k-2h-1}$ and the values of the Riemann zeta function at negative odd integers. This shows that $Z_k(-n)$ are rational numbers.

The results are extended to real projective spaces.

\end{abstract}

\maketitle

\section{Introduction}In \cite{5}, Minakshisundaram and   Pleijel defined a zeta function for any  compact Riemannian manifold $M$ in the following way. Consider the Laplace operator acting on $L^2$ functions on $M$. It is well known that the eigenvalues are nonnegative. The Minakshisundaram-Pleijel zeta function of $M$ is defined as
\begin{equation*}
\zeta_M(s)=\sum_{\lambda}\frac{1}{\lambda^s},
\end{equation*}where the sum is over all nonzero eigenvalues $\lambda$. This sum is shown to be convergent when $\text{Re}\,s>\text{dim}(M)/2$. Moreover, it can be analytically continued to the whole complex plane with at most simple poles at $s=\text{dim}(M)/2-n$, where $n$ is a nonnegative integer. The residues of the zeta function at the poles can be expressed in terms of metric invariants of the manifold $M$. However, for general manifold $M$, it is not easy to compute these invariants.

For the simplest compact manifold -- the unit circle $S^1$, the eigenvalues of the Laplace operator are  $n^2$, where $n\geq 0$. For each $n\geq 1$, $n^2$ is an eigenvalue with multiplicity 2. Therefore, the  Minakshisundaram-Pleijel zeta function of the unit circle $S^1$ is
\begin{equation*}
Z_1(s)=\sum_{n=1}^{\infty}\frac{2}{n^{2s}}=2\zeta(2s),
\end{equation*}where $\zeta(s)$ is the Riemann-zeta function. The Riemann zeta function has been actively under research. It is closely related to the distribution of primes.  It is   well-known that this function has trivial zeros on negative even integers. Its values at the negative odd integers are also well-known and they are rational numbers.

In general, the Minakshisundaram-Pleijel zeta function of a manifold plays an important role in studying the asymptotic behavior of the eigenvalues. It also appears extensively  in the study of theoretical physics especially in computing partition function and Casimir effect. The Minakshisundaram-Pleijel zeta function of $k$-dimensional spheres have been studied in a number of works \cite{1,2,8,7,6,5,3}. The starting point of this work is to give another much more intuitive proof for the formula derived in \cite{2}. We also study the values of the zeta functions at non-positive integers.

 \section{The Minakshisundaram-Pleijel zeta function of spheres}
Let $S^k$ be the $k$-dimensional unit sphere. It is well-known that for the Laplace operator acting on the $L^2$ functions on $S^k$, the  eigenvalues are
\begin{align*}
\lambda_n=n(n+k-1)
\end{align*}
  with multiplicities
\begin{align*}P_k(n)=\frac{(2n+k-1)(n+1)\ldots (n+k-2)}{(k-1)!}.
\end{align*}Here $n$ is a nonnegative integer. Therefore, the Minakshisundaram-Pleijel zeta function of $S^k$ is given by the Dirichlet series
\begin{equation*}
Z_k(s)=\sum_{n=1}^{\infty} \frac{P_k(n)}{\left[n(n+k-1)\right]^s}
\end{equation*}when $\displaystyle \text{Re}\,s>\frac{k}{2}$.

Recall that the Stirling numbers of the first kind $s_{n,k}$ are defined by
\begin{equation}\label{eq11_26_3}
x(x+1)\ldots(x+n-1)=\sum_{k=0}^{n}(-1)^{n+k}s_{n,k}x^k.
\end{equation}

In \cite{2}, Carletti and Bragadin used their results in \cite{1} and the properties of the Stirling numbers to prove the following theorem.

\begin{theorem}\label{theorem1}
For $k\geq 2$,
\begin{equation*}
Z_k(s)=\frac{1}{(k-1)!}\sum_{l=0}^{\infty} (-1)^l \begin{pmatrix} -s\\l\end{pmatrix}\left(\frac{k-1}{2}\right)^{2l}
\sum_{j=0}^{k-1}B_{k,j}\zeta\left(2s+2l-j;\frac{k+1}{2}\right),
\end{equation*}
where
\begin{equation}\label{eq11_25_1}
B_{k,j}=\sum_{p=0}^{k-j-1}(-1)^{k+j+1}\begin{pmatrix} j+p\\j\end{pmatrix}\left(\frac{k-1}{2}\right)^p\left(s_{k,j+p+1}+s_{k-1,j+p}\right),
\end{equation}and
$$ \zeta(s;a)=\sum_{n=0}^{\infty}\frac{1}{(n+a)^{s}}$$ is the Hurwitz zeta function.

\end{theorem}

They also proved that
\begin{theorem}\label{theorem2}
Let $B_{k,j}$ be defined by \eqref{eq11_25_1}. Then $B_{k,j}=0$ if $j=0$ or if $j=k-2h$ with $\displaystyle 1\leq h\leq \left[\tfrac{k}{2}\right]$.
\end{theorem}

In this work,  we reprove Theorem \ref{theorem1} and \ref{theorem2} in much simpler ways. In fact, we prove that
\begin{theorem}\label{theorem3}
For $k\geq 2$,
\begin{equation}\label{eq11_26_1}
Z_k(s)=\frac{1}{(k-1)!}\sum_{l=0}^{\infty} (-1)^l \begin{pmatrix} -s\\l\end{pmatrix}\left(\frac{k-1}{2}\right)^{2l}
\sum_{j=0}^{k-1}C_{k,j}\zeta\left(2s+2l-j;\frac{k+1}{2}\right),
\end{equation}
where the constants $C_{k,j}$ are defined so that
\begin{equation}\label{eq11_25_2}
P_k\left(x-\frac{k-1}{2}\right)=\frac{1}{(k-1)!}\sum_{j=0}^{k-1} C_{k,j} x^j.
\end{equation}
\end{theorem}

\begin{theorem}\label{theorem4}
The constants $C_{k,j}$ defined by \eqref{eq11_25_2} coincide  with the constants $B_{k,j}$ defined by \eqref{eq11_25_1}.
\end{theorem}

Now we give the proof of Theorem \ref{theorem3}.

\begin{proof} Using the definition \eqref{eq11_25_2} of $C_{k,j}$, we have
\begin{equation*}
P_k(n)=\frac{1}{(k-1)!}\sum_{j=0}^{k-1} C_{k,j} \left(n+\frac{k-1}{2}\right)^j.
\end{equation*}
On the other hand, since
$$\left[n(n+k-1)\right]^{-s}=\left[\left(n+\frac{k-1}{2}\right)^2-\left(\frac{k-1}{2}\right)^2\right]^{-s},$$ when $n\geq 1$, binomial expansion gives an absolutely convergent series
\begin{equation*}
\begin{split}
\left[n(n+k-1)\right]^{-s}
=&\sum_{l=0}^{\infty}(-1)^l\begin{pmatrix} -s\\l\end{pmatrix} \left(n+\frac{k-1}{2}\right)^{-2l-2s}\left(\frac{k-1}{2}\right)^{2l}.
\end{split}\end{equation*}The convergence is uniform for all $n\geq 1$.
Therefore,
\begin{equation*}\begin{split}
Z_k(s)=&\frac{1}{(k-1)!}\sum_{n=1}^{\infty} \sum_{l=0}^{\infty}(-1)^l\begin{pmatrix} -s\\l\end{pmatrix}\left(\frac{k-1}{2}\right)^{2l}\sum_{j=0}^{k-1} C_{k,j} \left(n+\frac{k-1}{2}\right)^{-2l-2s+j}\\
=&\frac{1}{(k-1)!}\sum_{l=0}^{\infty} (-1)^l \begin{pmatrix} -s\\l\end{pmatrix}\left(\frac{k-1}{2}\right)^{2l}
\sum_{j=0}^{k-1}C_{k,j}\zeta\left(2s+2l-j;\frac{k+1}{2}\right).
\end{split}\end{equation*}
\end{proof}

Compare to the proof given in \cite{2}, this proof we give here is much more succinct and clearly reflect the true nature of the formula \eqref{eq11_26_1}. Moreover, \eqref{eq11_25_2} gives a very straightforward way to compute the coefficients $C_{k,j}$, compare to the formula for $B_{k,j}$ given by \eqref{eq11_25_1} in terms of the Stirling numbers. However, to show the equivalence of our Theorem \ref{theorem3} with Theorem \ref{theorem1} proved in \cite{2}, we need to prove Theorem \ref{theorem4}.

Let us start with a lemma.
\begin{lemma}\label{lemma1}
\begin{equation}\label{eq11_26_5}
B_{k,j}=2\sum_{p=0}^{k-j-1}(-1)^{k+j+1}\begin{pmatrix} j+p-1\\j-1\end{pmatrix}\left(\frac{k-1}{2}\right)^ps_{k-1,j+p}.
\end{equation}
\end{lemma}
\begin{proof}
We need the following identity \cite{4}:
\begin{equation}\label{eq11_26_2}
s_{n+1,k}=s_{n,k-1}-ns_{n,k}
\end{equation}which can be easily proved using the fact that
\begin{equation}
x(x+1)\ldots(x+n-1)(x+n)=x^2(x+1)\ldots(x+n-1)+nx(x+1)\ldots(x+n-1).
\end{equation}
Using \eqref{eq11_26_2}, we have
\begin{equation*}
\begin{split}
B_{k,j}=\sum_{p=0}^{k-j-1}(-1)^{k+j+1}\begin{pmatrix} j+p\\j\end{pmatrix}\left(\frac{k-1}{2}\right)^p\left[2s_{k-1,j+p}-(k-1)s_{k-1,j+p+1}\right].
\end{split}\end{equation*}Splitting the two terms and re-index the second summation, we have
\begin{equation*}
\begin{split}
B_{k,j}=&2\sum_{p=0}^{k-j-1}(-1)^{k+j+1}\begin{pmatrix} j+p\\j\end{pmatrix}\left(\frac{k-1}{2}\right)^ps_{k-1,j+p}\\
&-\sum_{p=1}^{k-j-1}(-1)^{k+j+1}\begin{pmatrix} j+p-1\\j\end{pmatrix}\left(\frac{k-1}{2}\right)^{p-1}(k-1)s_{k-1,j+p} \\
=&2\sum_{p=0}^{k-j-1}(-1)^{k+j+1}\left[\begin{pmatrix} j+p\\j\end{pmatrix}-\begin{pmatrix} j+p-1\\j\end{pmatrix}\right]\left(\frac{k-1}{2}\right)^ps_{k-1,j+p}.
\end{split}
\end{equation*}
Eq. \eqref{eq11_26_5} then follows from the identity
$$\begin{pmatrix} j+p\\j\end{pmatrix}-\begin{pmatrix} j+p-1\\j\end{pmatrix}=\begin{pmatrix} j+p-1\\j-1\end{pmatrix}.$$
\end{proof}

Now we prove Theorem \ref{theorem4}.
\begin{proof}
By definition \eqref{eq11_25_2},
\begin{equation}\label{eq11_26_6}
\sum_{j=0}^{k-1}C_{k,j}x^j =2x\left(x-\frac{k-1}{2}+1\right)\left(x-\frac{k-1}{2}+2\right)\ldots\left(x-\frac{k-1}{2}+k-2\right).
\end{equation}
Notice   that from the definition \eqref{eq11_26_3}, it is easy to infer that
$$s_{n,0}=0$$ and
\begin{equation*}\begin{split}
(x+1)(x+2)\ldots(x+n-1)=&\sum_{k=0}^{n}(-1)^{n+k} s_{n,k}x^{k-1}.
\end{split}\end{equation*}
Therefore,
\begin{equation*}\begin{split}
\sum_{j=0}^{k-1}C_{k,j}x^j =&2x\sum_{p=0}^{k-1}(-1)^{p+k+1}s_{k-1,p}\left(x-\frac{k-1}{2}\right)^{p-1}\\
=&2x\sum_{p=0}^{k-1}(-1)^{p+k+1}s_{k-1,p}\sum_{j=1}^p\begin{pmatrix} p-1\\j-1\end{pmatrix} x^{j-1} (-1)^{p-j}\left( \frac{k-1}{2}\right)^{p-j}\\
=&2\sum_{j=1}^{k-1} \sum_{p=j}^{k-1} (-1)^{k+j+1} \begin{pmatrix} p-1\\j-1\end{pmatrix} \left( \frac{k-1}{2}\right)^{p-j}s_{k-1,p}x^j\\
=&2\sum_{j=1}^{k-1} \sum_{p=0}^{k-j-1} (-1)^{k+j+1} \begin{pmatrix} j+p-1\\j-1\end{pmatrix} \left( \frac{k-1}{2}\right)^{p}s_{k-1,j+p}x^j.
\end{split}\end{equation*}
Comparing the coefficients of $x^j$ on both sides shows that
\begin{equation*}
C_{k,j}=2\sum_{p=0}^{k-j-1} (-1)^{k+j+1} \begin{pmatrix} j+p-1\\j-1\end{pmatrix} \left( \frac{k-1}{2}\right)^{p}s_{k-1,j+p}.
\end{equation*}
Lemma \ref{lemma1} then proves that $C_{k,j}=B_{k,j}$.
\end{proof}

Now we give another proof of Theorem \ref{theorem2} which is much shorter and intuitive.
\begin{proof}
From Theorem \ref{theorem4} and \eqref{eq11_26_6}, we have
\begin{equation}\label{eq11_26_7}
\begin{split}
\sum_{j=0}^{k-1}B_{k,j}x^j =&2x\left(x-\frac{k-3}{2}\right)\left(x-\frac{k-5}{2}\right)\ldots\left(x+\frac{k-5}{2}\right)\left(x+\frac{k-3}{2}\right).
\end{split}
\end{equation}By putting $x=0$ on both sides, it follows immediately that $B_{k,0}=0$. Now when $k$ is odd, \eqref{eq11_26_7} gives
\begin{equation}\label{eq11_26_8}\begin{split}
\sum_{j=0}^{k-1}B_{k,j}x^j =&2x^2\left(x^2-\left[\frac{k-3}{2}\right]^2\right)\left(x^2-\left[\frac{k-5}{2}\right]^2\right)\ldots\left(x^2-1\right).
\end{split}\end{equation}
The right hand side is an even polynomial. Hence $B_{k,j}=0$ if $j$ is odd. Namely, $B_{k,k-2h}=0$ for $\displaystyle 1\leq h\leq\left[\tfrac{k}{2}\right]$.

When $k$ is even,
 \eqref{eq11_26_7} gives
\begin{equation}\label{eq11_26_9}\begin{split}
\sum_{j=0}^{k-1}B_{k,j}x^j =&2x \left(x^2-\left[\frac{k-3}{2}\right]^2\right)\left(x^2-\left[\frac{k-5}{2}\right]^2\right)\ldots\left(x^2-\left[\frac{1}{2}\right]^2\right).
\end{split}\end{equation}The right hand side is an odd polynomial. Hence $B_{k,j}=0$ if $j$ is even. Namely, we also have $B_{k,k-2h}=0$ for $\displaystyle 1\leq h\leq\left[\tfrac{k}{2}\right]$.
\end{proof}

It is obvious that it is easier to compute the coefficients $B_{k,j}$ using the identities \eqref{eq11_26_8} and \eqref{eq11_26_9} rather than using the expression \eqref{eq11_25_1} in terms of Stirling numbers. In fact, $B_{k,j}$ can be computed recursively.

\begin{theorem}\label{theorem7}For $k\geq 2$,
$B_{k,j}$ is nonzero if $j\neq 0$ and $k$ and $j$ does not have the same parity. In such cases,  $B_{k,j}$ can be computed recursively as follows:
\begin{align}
B_{2,1}=&2,\nonumber\\
B_{3,2}=&2,\nonumber\\
\label{eq11_26_10} B_{k,j}=&B_{k-2,j-2}-\left(\frac{k-3}{2}\right)^2B_{k-2,j},\quad k\geq 4.
\end{align}
In particular,
$B_{k,k-1}=2$ for all $k$.
\end{theorem}
\begin{proof}
The first statement is a rephrasing of Theorem \ref{theorem2}.
For the second statement, notice that
\begin{equation*}\begin{split}
P_2(n)=&2n+1,\\
P_3(n)=&2(n+1)^2.\end{split}
\end{equation*}Therefore,
\begin{equation*}
\begin{split}
P_2\left(x-\frac{1}{2}\right)=&2x,\\
P_3\left(x-1\right)=&2x^2.
\end{split}
\end{equation*}From these it follows that $B_{2,1}=2$ and $B_{3,2}=2$.
For $k\geq 4$, it follows from \eqref{eq11_26_8} and \eqref{eq11_26_9}  that
\begin{equation}
\sum_{j=1}^{k-1}B_{k,j}x^j=\left(x^2-\left[\frac{k-3}{2}\right]^2\right)\sum_{j=1}^{k-3}B_{k-2,j}x^j.
\end{equation}Eq. \eqref{eq11_26_10} then follows by comparing the coefficients of $x^j$ on both sides.
\end{proof}

\vspace{0.2cm}
\begin{corollary}\label{co1}~
\begin{enumerate}

\item[(i)]If $k\geq 2$ is odd, $B_{k,j}$ is an integer.

\item[(ii)] If $k\geq 2$ is even, $2^{k-2}B_{k,j}$ is an integer.
\end{enumerate}
\end{corollary}
\begin{proof}
This follows from Theorem \ref{theorem7} by induction on odd $k$ and even $k$ respectively.
\end{proof}

There are some interesting identities involving the constants $B_{k,j}$ that have been proved  in complicated ways in \cite{2}. Using the definition \eqref{eq11_25_2}, we can prove them in much simpler ways:

\begin{theorem}\label{theorem12}
For $k\geq 2$,
\begin{equation}\label{eq11_28_1}
\sum_{j=1}^{k-1}B_{k,j}\left(\frac{k-1}{2}\right)^j=(k-1)!.
\end{equation}
Moreover,
\begin{enumerate}
\item[(i)] If $k\geq 2$ is an odd integer and $\displaystyle p=1, 2,\ldots,\frac{k-3}{2}$,
\begin{equation}\label{eq11_28_2}
\sum_{j=1}^{k-1}B_{k,j}p^j=0.
\end{equation}

\item[(ii)] If $k\geq 2$ is an even integer and $\displaystyle p=\frac{1}{2}, \frac{3}{2},\ldots,\frac{k-3}{2}$,
\begin{equation}\label{eq11_28_2}
\sum_{j=1}^{k-1}B_{k,j}p^j=0.
\end{equation}
\end{enumerate}
\end{theorem}
\begin{proof}
All these three identities can be proved by substituting an appropriate value of $x$ in \eqref{eq11_26_7}. 
\end{proof}

As is observed in \cite{2}, Theorem \ref{theorem1} shows that $Z_k(s)$ is a meromorphic function. All the poles are simple and they can only appear at $$s=\frac{k}{2}-n,\quad n=0,1,2,\ldots.$$ Moreover,

\begin{theorem}\label{theorem8}
If $k\geq 2$ and  $n$ is a nonnonnegative integer,
\begin{equation}\label{eq11_27_1}\begin{split}
&\text{Res}_{s=\frac{k}{2}-n}Z_k(s)\\=&\frac{1}{2(k-1)!}\sum_{h=0}^{\min\left\{n,\left[\frac{k-2}{2}\right]\right\}}(-1)^{n-h}\left(\frac{k-1}{2}\right)^{2n-2h}\begin{pmatrix}\displaystyle  n-\frac{k}{2}\\n-h\end{pmatrix}B_{k,k-2h-1}.\end{split}
\end{equation}
In particular,
\begin{equation}\label{eq11_27_2}
\text{Res}_{s=\frac{k}{2} }Z_k(s)=\frac{1}{(k-1)!}.
\end{equation}
\end{theorem}
\begin{proof}
The proof is the same as that given in \cite{2}. Notice that
\begin{equation*}
\begin{pmatrix} -s\\l\end{pmatrix}=(-1)^l\frac{s(s+1)\ldots(s+l-1)}{l!}
\end{equation*}is a polynomial in $s$. The function $\displaystyle \zeta\left(2s+2l-j;\frac{k+1}{2}\right)$ has a simple pole at $\displaystyle s=\frac{j+1-2l}{2}$ with residue $1/2$. Eq. \eqref{eq11_27_1} is then established using Theorem \ref{theorem1} and Theorem \ref{theorem2}. Taking $n=0$ in \eqref{eq11_27_1}, then $h=0$ and $B_{k,k-1}=2$ implies \eqref{eq11_27_2}.
\end{proof}

\begin{remark}
The formula given in \cite{2} is twice the formula \eqref{eq11_27_1}. The discrepancy is because the authors in \cite{2} have taken the residue of $\displaystyle \zeta\left(2s+2l-j;\frac{k+1}{2}\right)$ at $\displaystyle s=\frac{j+1-2l}{2}$ to be 1, but it should be $1/2$.

\end{remark}

\begin{corollary}
If $k\geq 2$ is even and $\displaystyle n\geq \frac{k}{2}$ is an integer, then
\begin{equation*}
\text{Res}_{s=\frac{k}{2}-n}Z_k(s)=0.
\end{equation*}
\end{corollary}
\begin{proof}
This follows from \eqref{eq11_27_1} using the fact that when $k$ is even and $\displaystyle  n\geq \frac{k}{2}$, $\displaystyle  n-\frac{k}{2}$ is a nonnegative integer. Moreover, $\displaystyle n-h\geq  n-\frac{k}{2}$ since $\displaystyle
h\leq \frac{k-2}{2}$. Therefore,
\begin{equation*}
\begin{pmatrix}\displaystyle  n-\frac{k}{2}\\n-h\end{pmatrix}=0.
\end{equation*}
\end{proof}

\begin{corollary}
If $k\geq 2$ is even, $Z_k(s)$ is a meromorphic function with at most simple poles at $\displaystyle s=1,\ldots, \frac{k}{2}$.
\end{corollary}

Now we prove the following result stated in \cite{2}:
\begin{theorem}\label{theorem6}
If $k\geq 3$ is  odd and $n$ is a positive integer, then
$$Z_k(-n)=0.$$
\end{theorem}
\begin{proof}
From Theorem \ref{theorem1} and  Theorem \ref{theorem2}, we have
\begin{equation*}\begin{split}
Z_k(-n)=&\frac{1}{(k-1)!}\sum_{l=0}^{n} (-1)^l \begin{pmatrix} n\\l\end{pmatrix}\left(\frac{k-1}{2}\right)^{2l}
\sum_{h=0}^{\frac{k-3}{2}}B_{k,k-2h-1}\zeta\left(-2n+2l-k+2h+1;\frac{k+1}{2}\right).
\end{split}\end{equation*}
Now since the Riemann zeta function vanishes at negative even integers, we find that for $0\leq l\leq n$ and $\displaystyle 0\leq h\leq\tfrac{k-3}{2}$, 
\begin{equation*}\begin{split}
\zeta\left(-2n+2l-k+2h+1;\frac{k+1}{2}\right)=&\zeta(-2n+2l-k+2h+1)-\sum_{p=1}^{\frac{k-1}{2}}p^{2n-2l+k-2h-1}\\
=&-\sum_{p=1}^{\frac{k-1}{2}}p^{2n-2l+k-2h-1}.
\end{split}\end{equation*}
 Therefore,
\begin{equation*}\begin{split}
Z_k(-n)=&-\frac{1}{(k-1)!}\sum_{l=0}^{n} (-1)^l \begin{pmatrix} n\\l\end{pmatrix}\left(\frac{k-1}{2}\right)^{2l}
\sum_{p=1}^{\frac{k-1}{2}}p^{2n-2l}\sum_{h=1}^{\frac{k-3}{2}}B_{k,k-2h-1}p^{k-2h-1}.
\end{split}\end{equation*}
Theorem \ref{theorem12} then shows that 
\begin{equation*}\begin{split}
Z_k(-n)=&-\frac{1}{(k-1)!}\sum_{l=0}^{n} (-1)^l \begin{pmatrix} n\\l\end{pmatrix}\left(\frac{k-1}{2}\right)^{2l}
\left(\frac{k-1}{2}\right)^{2n-2l}\sum_{h=1}^{\frac{k-3}{2}}B_{k,k-2h-1}\left(\frac{k-1}{2}\right)^{k-2h-1}\\
=&- \sum_{l=0}^{n} (-1)^l \begin{pmatrix} n\\l\end{pmatrix}\left(\frac{k-1}{2}\right)^{2n}\\
=&0.
\end{split}\end{equation*}
The last equality follows from the fact that
$$\sum_{l=0}^{n}  (-1)^l \begin{pmatrix} n\\ l \end{pmatrix}=0$$when $n\geq 1$.
\end{proof}
\begin{remark}
When $k=1$, we have $$Z_1(s)=2\zeta(2s).$$Hence Theorem \ref{theorem6} can be considered as the generalization of the fact that $\zeta(-2n)=0$ if $n$ is a positive integer.
\end{remark}

As a generalization of
$$\zeta(0)=-\frac{1}{2},$$ we have
\begin{theorem}\label{theorem9}
If $k\geq 3$ is an odd integer, then
\begin{equation}
Z_k(0)=-1.
\end{equation}
\end{theorem}
\begin{proof}
Notice that when $s=0$,
\begin{align*}
\begin{pmatrix}-s\\l\end{pmatrix}=\begin{cases} 1,\quad \text{if}\quad l=0,\\
0,\quad\text{if} \quad l\geq 1\end{cases}.
\end{align*}As in the proof of Theorem \ref{theorem6}, we find that when $k$ is odd,
\begin{equation*}\begin{split}
Z_k(0)=&-\frac{1}{(k-1)!}
\sum_{p=1}^{\frac{k-1}{2}} \sum_{h=1}^{\frac{k-3}{2}}B_{k,k-2h-1}p^{k-2h-1}\\
=&-\frac{1}{(k-1)!}
 \sum_{h=1}^{\frac{k-3}{2}}B_{k,k-2h-1}\left(\frac{k-1}{2}\right)^{k-2h-1}\\
=&-1.
\end{split}\end{equation*}
\end{proof}

Finally, we consider the value of $Z_k(s)$ at $s=-n$ when $k$ is an even integer and $n$ is a nonnegative integer:
\begin{theorem}\label{theorem10}
When $k$ is even,
\begin{equation}\label{eq11_27_5}
Z_k(0)=\frac{1}{(k-1)!}\sum_{h=0}^{\frac{k-2}{2}}B_{k,k-2h-1}\left(\left[2^{2h+1-k}-1\right]\zeta\left(2h+1-k\right)+\frac{1}{k-2h}\left(\frac{k-1}{2}\right)^{k-2h}\right)-1.
\end{equation}
Moreover, if $n$ is a positive integer,
\begin{equation}\label{eq11_27_6}\begin{split}
&Z_k(-n)\\=&\frac{1}{(k-1)!}\sum_{h=0}^{\frac{k-2}{2}}B_{k,k-2h-1}\left(\sum_{l=0}^n (-1)^l\begin{pmatrix} n\\l\end{pmatrix}\left(\frac{k-1}{2}\right)^{2l}\left[2^{2h+1+2l-2n-k}-1\right]\zeta\left(2h+1+2l-2n-k\right)\right.\\
&\left.+\frac{(-1)^nn!}{2}\frac{\Gamma\left( \frac{k-2h}{2}  \right)}{\Gamma\left( \frac{k-2h}{2}+n+1\right)}\left(\frac{k-1}{2}\right)^{k-2h+2n}\right).
\end{split}\end{equation}
\end{theorem}
\begin{proof}
From Theorem \ref{theorem1} and Theorem \ref{theorem2}, we have
\begin{equation}\begin{split}
Z_k(s)=&\frac{1}{(k-1)!}\sum_{l=0}^{\infty}   \frac{s(s+1)\ldots(s+l-1)}{l!}\left(\frac{k-1}{2}\right)^{2l}
\\&\times \sum_{h=0}^{\frac{k-2}{2}}B_{k,k-2h-1}\zeta\left(2s+2l+2h-k+1;\frac{k+1}{2}\right).
\end{split}
\end{equation}

For $s=-n$, where $n$ is a nonnegative integer, the nonzero terms come from $0\leq l\leq n$ and $\displaystyle l=\frac{k-2h}{2}+n$, $\displaystyle 0\leq h\leq \frac{k-2}{2}$. Notice that
\begin{equation*}
 \frac{s(s+1)\ldots(s+l-1)}{l!}= \frac{\Gamma(s+l)}{\Gamma(l+1)}\frac{s(s+1)\ldots(s+n-1)}{\Gamma(s+n+1)}(s+n).
\end{equation*}
Therefore,
\begin{equation*}\begin{split}
&Z_k(-n)\\=&\frac{1}{(k-1)!}\sum_{l=0}^{n}  (-1)^l \begin{pmatrix} n\\ l \end{pmatrix}\left(\frac{k-1}{2}\right)^{2l}
  \sum_{h=0}^{\frac{k-2}{2}}B_{k,k-2h-1}\zeta\left(-2n+2l+2h-k+1;\frac{k+1}{2}\right)\\
  &+\frac{1}{(k-1)!}\sum_{h=0}^{\frac{k-2}{2}}\frac{(-1)^nn!}{2}\frac{\Gamma\left(  \frac{k-2h}{2}\right)}{\Gamma\left( \frac{k-2h}{2}+n+1\right)}
  \left(\frac{k-1}{2}\right)^{k-2h+2n}B_{k,k-2h-1}.
\end{split}
\end{equation*}
Now
\begin{align*}
&\zeta\left( -2n+2l+2h-k+1;\frac{k+1}{2}\right)\\=&\zeta\left(-2n+2l+2h-k+1;\frac{1}{2}\right)-\sum_{p=\frac{1}{2},\frac{3}{2},\ldots,\frac{k-1}{2}}p^{k-2h-1}p^{2n-2l}.
\end{align*}
Theorem \ref{theorem12} shows that
\begin{equation*}
\begin{split}
&\frac{1}{(k-1)!}\sum_{l=0}^{n}  (-1)^l \begin{pmatrix} n\\ l \end{pmatrix}\left(\frac{k-1}{2}\right)^{2l}
  \sum_{h=0}^{\frac{k-2}{2}}B_{k,k-2h-1}\zeta\left(-2n+2l+2h-k+1;\frac{k+1}{2}\right)\\
  =&\frac{1}{(k-1)!}\sum_{l=0}^{n}  (-1)^l \begin{pmatrix} n\\ l \end{pmatrix}\left(\frac{k-1}{2}\right)^{2l}
  \sum_{h=0}^{\frac{k-2}{2}}B_{k,k-2h-1} \zeta\left(-2n+2l+2h-k+1;\frac{1}{2}\right)\\&- \sum_{l=0}^{n}  (-1)^l \begin{pmatrix} n\\ l \end{pmatrix}\left(\frac{k-1}{2}\right)^{2n}.
\end{split}
\end{equation*}
However,
\begin{equation*}
\sum_{l=0}^{n}  (-1)^l \begin{pmatrix} n\\ l \end{pmatrix}=\begin{cases} 1,\quad \text{if}\quad n=0\\
0,\quad \text{if} \quad n\geq 1\end{cases}.
\end{equation*}Moreover,
\begin{equation}\label{eq11_27_10}
\zeta\left(s;\frac{1}{2}\right)=(2^s-1)\zeta(s).
\end{equation}
The assertion of the theorem follows.
\end{proof}

\vspace{0.5cm}
\begin{corollary}
If $k\geq 2$ and  $n$ is a nonnegative integer,
$Z_k(-n)$ is a rational number.
\end{corollary}
\begin{proof}
This follows from Theorem \ref{theorem6}, Theorem \ref{theorem9}, Theorem \ref{theorem10}, Corollary \ref{co1} and the fact that $\zeta(-j)$ is a rational number when $j$ is an odd positive integer.
\end{proof}

\bigskip
\section{The Minakshisundaram-Pleijel zeta function of real projective spaces}

The  Minakshisundaram-Pleijel zeta function of the $k$-dimensional real projective space $P^k$ is given by
\begin{equation*}
L_k(s)=\sum_{n=1}^{\infty} \frac{P_k(2n)}{[2n(2n+k-1)]^s}.
\end{equation*}
From \eqref{eq11_25_2} and Theorem \ref{theorem4}, we find that
\begin{equation*}
\begin{split}
P_k(2n)=&\frac{1}{(k-1)!}\sum_{j=0}^{k-1}C_{k,j}\left(2n+\frac{k-1}{2}\right)^j\\
=&\frac{1}{(k-1)!}\sum_{j=0}^{k-1}2^jB_{k,j}\left( n+\frac{k-1}{4}\right)^j.
\end{split}
\end{equation*}
As the proof of Theorem \ref{theorem3}, we immediately obtain
\begin{theorem}\label{theorem3p}
For $k\geq 2$,
\begin{equation}
L_k(s)=\frac{2^{-2s}}{(k-1)!}\sum_{l=0}^{\infty} (-1)^l \begin{pmatrix} -s\\l\end{pmatrix}\left(\frac{k-1}{4}\right)^{2l}
\sum_{j=0}^{k-1}2^jB_{k,j}\zeta\left(2s+2l-j;\frac{k+3}{4}\right).
\end{equation}

\end{theorem}

From this, it follows that
\begin{theorem}\label{theorem11}
If $k\geq 2$, $L_k(s)$ at most has simple poles at $\displaystyle s=\frac{k}{2}-n$, where  $n$ is a nonnonnegative integer. Moreover,
\begin{equation}\begin{split}
&\text{Res}_{s=\frac{k}{2}-n}L_k(s)\\=&\frac{1}{ 4(k-1)!}\sum_{h=0}^{\min\left\{n,\left[\frac{k-2}{2}\right]\right\}}(-1)^{n-h}\left(\frac{k-1}{2}\right)^{2n-2h}\begin{pmatrix}\displaystyle  n-\frac{k}{2}\\n-h\end{pmatrix}B_{k,k-2h-1}.\end{split}
\end{equation}
In particular,
\begin{equation}
\text{Res}_{s=\frac{k}{2} }L_k(s)=\frac{1}{2(k-1)!}.
\end{equation}Moreover, this implies that when $k$ is even, $L_k(s)$ at most has simple poles at $\displaystyle s=1,  \ldots,\frac{k}{2}$.
\end{theorem}

 Compare Theorem \ref{theorem11} to Theorem \ref{theorem8}, we notice that
 \begin{equation*}
 \text{Res}_{s=\frac{k}{2}-n}L_k(s)=\frac{1}{2} \text{Res}_{s=\frac{k}{2}-n}Z_k(s).
 \end{equation*}

 Now,   Theorem \ref{theorem12} implies that
\begin{equation*}\begin{split}
 \sum_{j=0}^{k-1}2^jB_{k,j}\left(\frac{k-1}{4}\right)^j=& (k-1)!.
 \end{split}\end{equation*}Moreover, when $k$ is odd, and
$$ p=\frac{1}{2}, \frac{2}{2},  \ldots, \frac{k-5}{4},\frac{k-3}{4},$$then
 \begin{equation*}\begin{split}
 \sum_{j=0}^{k-1}2^jB_{k,j}p^j=& 0
 \end{split}\end{equation*}
 Together with the identity \eqref{eq11_27_10}, we find as in the proof of Theorem \ref{theorem6} and Theorem \ref{theorem9} that

\begin{theorem}
If $k\geq 2$ is an odd integer,
\begin{align*}
L_k(0)=-1.
\end{align*}Moreover, if $n$ is a positive integer,
\begin{align*}
L_k(-n)=0.
\end{align*}
\end{theorem}

\bigskip

\bibliographystyle{amsplain}

\end{document}